\documentclass[11pt]{amsart}
\usepackage{amscd,amssymb,amsopn,amsmath,amsthm,mathrsfs,graphics,amsfonts,enumerate,verbatim,calc
}
\usepackage[all,cmtip]{xy}
\usepackage[all]{xy}
\usepackage{tikz}
\usepackage{lscape}
\usetikzlibrary{matrix,arrows,decorations.pathmorphing}
\usetikzlibrary {positioning}
\definecolor {processblue}{cmyk}{0.96,0,0,0}

\usepackage{ dsfont }
\usepackage{color}

\usepackage[OT2,OT1]{fontenc}
\newcommand\cyr{%
\renewcommand\rmdefault{wncyr}%
\renewcommand\sfdefault{wncyss}%
\renewcommand\encodingdefault{OT2}%
\normalfont
\selectfont}
\DeclareTextFontCommand{\textcyr}{\cyr}

\usepackage{amssymb,amsmath}

\DeclareFontFamily{OT1}{rsfs}{}
\DeclareFontShape{OT1}{rsfs}{n}{it}{<-> rsfs10}{}
\DeclareMathAlphabet{\mathscr}{OT1}{rsfs}{n}{it}

\numberwithin{equation}{section}

\newtheorem{theorem}{Theorem}[section]
\newtheorem{lemma}[theorem]{Lemma}
\newtheorem{prop}[theorem]{Proposition}
\newtheorem{cor}[theorem]{Corollary}

\theoremstyle{definition}
\newtheorem{defn}[theorem]{Definition}
\newtheorem{remark}[theorem]{Remark}
\theoremstyle{remark}

\newtheorem{example}[theorem]{Example}

\newcommand{\lk}{\operatorname{lk}}
\newcommand{\Ass}{\operatorname{Ass}}

\newcommand{\pol}{\operatorname{pol}}

\newcommand{\Spec}{\operatorname{Spec}}
\newcommand{\cd}{\operatorname{cd}}

\newcommand{\reg}{\operatorname{reg}}

\newcommand{\height}{\operatorname{ht}}
\newcommand{\pd}{\operatorname{pd}}

\newcommand{\Ext}{\operatorname{Ext}}
\newcommand{\Supp}{\operatorname{Supp}}
\newcommand{\Tor}{\operatorname{Tor}}

\newcommand{\depth}{\operatorname{depth}}

\begin{document}
\title[A Generalized Serre's Condition]{A Generalized Serre's Condition}

\author{Brent Holmes}
\address{Department of Mathematics\\
University of Kansas\\
Lawrence, KS 66045-7523 USA}
\email{brentholmes@ku.edu}
\date{\today}



\begin{abstract}
Throughout, let $R$ be a commutative Noetherian ring.  A ring $R$ satisfies Serre's condition $(S_{\ell})$ if for all $\mathfrak{p} \in \Spec R,$ $\depth R_{\mathfrak{p}} \geq \min \{ \ell , \dim R_{\mathfrak{p}} \}$. Serre's condition has been a topic of expanding interest. In this paper, we examine a generalization of Serre's condition $(S_{\ell}^j)$. We say a ring satisfies $(S_{\ell}^j)$ when $\depth R_{\mathfrak{p}} \geq \min \{ \ell , \dim R_{\mathfrak{p}} -j \}$ for all $\mathfrak{p} \in \Spec R$. We prove generalizations of results for rings satisfying Serre's condition.
\end{abstract}

\maketitle

\section{Introduction}

Let $R$ be a commutative Noetherian ring.  Recall Serre's condition $(S_\ell)$.

\begin{defn}
A ring $R$ satisfies Serre's condition $(S_{\ell})$ if for all $\mathfrak{p} \in \Spec R,$ \[\depth R_{\mathfrak{p}} \geq \min \{\ell , \dim R_{\mathfrak{p}} \} .\]
\end{defn}

For a $d$-dimensional ring, being Cohen-Macaulay is equivalent to satisfying $(S_d)$.  A multitude of authors have examined Serre's condition, and the popularity of the topic has continued to grow (cf. \cite{DH16, HT11, MT09, PF14, Te07, Ya00}).  Serre's condition, like the Cohen-Macaulay condition, ties homological properties to geometric properties.  We define a generalization of Serre's condition, which also links homological properties to geometric properties.

\begin{defn}
A ring $R$ satisfies $(S_{\ell}^j)$ property if for all $\mathfrak{p} \in \Spec R,$ \[\depth R_{\mathfrak{p}} \geq \min \{ \ell , \dim R_{\mathfrak{p}} - j \} .\]
\end{defn}

In this paper, we examine results about $(S_\ell)$ from a variety of sources in the literature.  We prove generalizations for the $(S_{\ell}^j)$ property.  The following will be used implicitly in the paper when appropriate.

\begin{prop}
Let $\phi: R \rightarrow S$ be a faithfully flat homomorphism of Noetherian rings.  Then if $S$ satisfies $(S_\ell^j)$ then so does $R$.
\end{prop}

\begin{proof}
Let $\mathfrak{p} \in \Spec R$.  Since $\phi$ is faithfully flat, there exists $\mathfrak{q} \in \Spec S$ such that $\mathfrak{p} = \mathfrak{q} \cap R$.  From \cite{BH98}, $\dim R_{\mathfrak{p}}=\dim S_{\mathfrak{q}}$ and $\depth R_{\mathfrak{p}}=\depth S_{\mathfrak{q}}$.  Thus $\depth R_{\mathfrak{p}} = \depth S_{\mathfrak{q}} \geq \min \{ n,\dim S_{\mathfrak{q}} - j \} = \min \{ n,\dim R_{\mathfrak{p}}-j \}$.  Thus $R$ satisfies $(S_\ell^j)$.
\end{proof}

\begin{remark}
We shall also make use of the fact that for a Stanley-Reisner ring $R$ and its associated simplicial complex $\Delta$, localizing $R$ at a prime $\mathfrak{p}$ generated by variables yields the same ring as the following process.  Localize $\Delta$ at a face $F$ generated by the variables which are not generators of $\mathfrak{p}$.  Then localize the Stanley-Reisner ring of $\lk _\Delta F$ at it's unique maximal homogeneous ideal.  Throughout the paper, we shall use $R_\mathfrak{p}$ and the Stanley-Reisner ring of $\lk_\Delta F$ interchangeably for convenience and brevity.
\end{remark}



We now describe the organization of the paper.

In Section $2$, we prove an equivalence between rings satisfying $(S_{\ell}^j)$ and the support of $\Ext$ functors.  In this section we consider rings of the form $R=S/I$ where $S$ is an $n$-dimensional polynomial ring and $I$ is a homogeneous ideal or $S$ is an $n$-dimensional complete regular local ring and $I$ is an ideal of $S$.

In Section $3$, we examine theorems from the literature which bound cohomological dimension when $R$ satisfies $(S_2)$ or $(S_3)$. We prove generalizations of these bounds that apply when $R$ is equidimensional and satisfies $(S_2^j)$ or $(S_3^j)$. Using these bounds, we bound projective dimension of $R$ when $\mathfrak{a}$ is a pure square-free monomial ideal and $R$ satisfies $(S_2^j)$ or $(S_3^j)$.  In this section we will consider $S$ to be an $n$-dimensional regular local ring containing a field, $\mathfrak{a}$ to be an ideal of $S$ of pure height and our ring to be $R=S/\mathfrak{a}$.

In Section $4$, we consider the Hochster-Huneke graph of $R$ (denoted by $G(R)$) where $R$ is a local ring or a quotient of a polynomial ring and a homogeneous ideal.  It is known that a Stanley-Reisner ring satisfies $(S_2)$ if and only if every localization of $R$ at a prime has a connected Hochster-Huneke graph \cite{Ku08, Ho16}.  We create a generalization of the Hochster-Huneke graph. We show that every localization of $R$ at a prime has a connected generalized Hochster-Huneke graph if and only if $R$ satisfies $(S_2^j)$.

In Section $5$, we expand upon a result of Yanagawa \cite{Ya00} which states that a Stanley-Reisner ring $S/I$ satisfies $(S_{\ell})$ if and only if the Alexander dual of $I$ satisfies a specific homological condition.  We combine this result with Theorem 3.4 to prove a bound on regularity of pure square-free monomial ideals.

Given a Stanley-Reisner ring $R$ with simplicial complex $\Delta$, Reisner's criterion provides a method to check the Cohen-Macaulayness of $R$ by examining the reduced homology groups of $\Delta$ \cite{Re76}. Terai made an analogous theorem for describing whether $R$ satisfies $(S_\ell)$ \cite{Te07}. In Section $6$, we generalize Terai's result by giving an equivalent condition for $(S_\ell^j)$ using homology of links in the equidimensional case.

In section $7$, we examine monomial ideals of polynomial rings over a field.  Herzog, Takayama, and Terai \cite{HT05} proved for a monomial ideal $I$ and a polynomial ring $S$, $S/I$ being Cohen-Macaulay implies $S/\sqrt{I}$ is Cohen-Macaulay. We present an analogous theorem showing $S/I$ satisfying $(S_\ell ^j)$ implies $S/\sqrt{I}$ satisfies $(S_\ell ^j)$. 


In section $8$, we prove a theorem relating $i$-skeletons of simplicial complexes and the $(S_\ell^j)$ property.  We discuss $i$-skeletons' importance with respect to depth.

Unless otherwise stated, $\mathbb{K}$ will be a field, $S=\mathbb{K}[x_1,x_2,...,x_n]$, and $I$ will be an ideal of $S$.

\section{An Equivalent Functorial Condition}

In this section, we characterize generalized Serre's condition as a homological condition.  This is an extension of a well-known theorem (see for example \cite[Proposition 2.6]{DH16}).

A catenary noetherian local ring satisfying $(S_2)$ is equidimensional by \cite[Remark 2.4.1]{Ha62}.  This is not true for $(S_2^j)$, $j>0$, and thus we will need the following notation. 

Given a ring $R=S/I$, let $Q_I$ be a minimal prime of $I$ of smallest height.  Given a prime $\mathfrak{p}$ which contains $I$, let $Q_{\mathfrak{p}}$ be a minimal prime of $I$ contained in $\mathfrak{p}$ with smallest height.  Let $\alpha_{\mathfrak{p}} = \height Q_{\mathfrak{p}} - \height Q_I$.  We note that $\alpha_{\mathfrak{p}} = 0$ for all $\mathfrak{p}$ when $R$ is equidimensional.

\begin{theorem}
Let $S$ be an $n$-dimensional polynomial ring with maximal homogenous ideal $\mathfrak{m}$ and let $I$ be a homogeneous ideal or let $S$ be an $n$-dimensional complete regular local ring with maximal ideal $\mathfrak{m}$ and let $I$ be an ideal of $S$.  Let  $R=S/I$ and let $d=\dim R$.  Then $R$ satisfies $(S_{\ell}^j)$ if and only if for all $\mathfrak{p} \in \Spec S$ containing $I$ with $\height \mathfrak{p} < n-i+\ell$, we have that $\mathfrak{p} \notin \Supp \Ext _S ^{n-i} (R,S)$ for all $i = 0,...,d-j-1-\alpha_\mathfrak{p}$.
\end{theorem}

\begin{proof}
Suppose $R$ satisfies $(S_{\ell}^j)$. If $i < \ell \leq \depth R$ then $H_{\mathfrak{m}}^i(R) = 0 = \Ext_S^{n-i}(R,S)$
, and we are finished.
Otherwise, let us take $\mathfrak{p} \in \Spec S$ such that $I \subseteq \mathfrak{p}$. Let $h$ be the height of $\mathfrak{p}$. If $h < n-i+\ell$, then
\[ \Ext _S^{n-i}(R,S)_\mathfrak{p} \cong \Ext _{S_\mathfrak{p}}^{n-i}(R_\mathfrak{p},S_\mathfrak{p}) = 0 \quad \textrm{ for all } i = 0,...,d-j-1-\alpha_\mathfrak{p}. \]
 This is true because $n-i = h-(h-n+i)$ and $\dim R_\mathfrak{p} - j = h-n+d-j-\alpha_\mathfrak{p} > h-n+i$.
So $\mathfrak{p} \notin \Supp \Ext _S ^{n-i} (R,S)$ for $i = 0,...,d-j-1-\alpha_\mathfrak{p}$ whenever $\height \mathfrak{p} < n-(i-\ell)$.

Suppose $\mathfrak{p} \notin \Supp \Ext _S ^{n-i} (R,S)$ for $i = 0,...,d-j-1-\alpha_\mathfrak{p}$ whenever $\height \mathfrak{p} < n-(i-\ell)$. Let $V_h$ be the set of prime ideals of $S$ of height $h$ containing $I$ for $h=n-d,...,n$.  Henceforth, let us use $c$ to denote $n-d$.
 
Then we have the following equivalences:

\[ R \textrm{ satisfies } (S_{\ell}^j) \quad \Leftrightarrow \]

\[ \depth R_\mathfrak{p} \geq \min \{ \ell, h-c-j-\alpha_\mathfrak{p} \} =: b \quad \textrm{ for all } h=c,...,n  \textrm{ for all } \mathfrak{p} \in V_h \quad \Leftrightarrow \]

\[ H_{\mathfrak{p}S_\mathfrak{p}}^i(R_\mathfrak{p})=0 \quad \textrm{ for all } h=c,...,n \textrm{ for all } \mathfrak{p} \in V_h, \textrm{ for all } i < b \quad \Leftrightarrow \]

\[ \Ext _{S_\mathfrak{p}}^{h-i}(R_\mathfrak{p},S_\mathfrak{p})=0 \textrm{ for all } h=c,...,n \textrm{ for all } \mathfrak{p} \in V_h, \textrm{ for all } i < b  \quad \Leftrightarrow \]

\[\mathfrak{p} \notin \Supp \Ext_S^{h-i}(R,S)  \textrm{ for all } h=c,...,n \textrm{ for all } \mathfrak{p} \in V_h, \textrm{ for all } i < b.\]

When $i < b \leq h-c-j-\alpha_\mathfrak{p}$, $n-h+i < n-c-j-\alpha_\mathfrak{p} = d-j-\alpha_\mathfrak{p}$.

From the assumption, for all $ n-h+i < d-j-\alpha_\mathfrak{p}$, $\mathfrak{p} \notin \Supp \Ext _S^{n-(n-h+i)}(R,S)$ whenever $h \leq n-(n-h+i-\ell)$; thus,
\[\mathfrak{p} \notin \Supp \Ext _S ^{h-i} (R,S) = \Supp \Ext _S ^{n-(n-h+i)}(R,S) \textrm{ for all } i < b. \]

\end{proof}

The result can be stated in a cleaner way when the ring in question is equidimensional.  In the equidimensional case, the above proof can be adjusted to give:

\begin{cor}
Let $S$ be an $n$-dimensional polynomial ring with maximal homogenous ideal $\mathfrak{m}$ and let $I$ be a homogeneous ideal or let $S$ be an $n$-dimensional complete regular local ring with maximal ideal $\mathfrak{m}$ and let $I$ be an ideal of $S$.  Let  $R=S/I$ be an equidimensional ring.  Then $R$ satisfies $(S_{\ell}^j)$ if and only if $\dim \Ext _S ^{n-i} (R,S) \leq i-\ell$ for all $i = 0,...,d-j-1$.
\end{cor}

\section{Bounds on Cohomological Dimension}

Local cohomology is widely important in the fields of commutative algebra and algebraic geometry. The vanishing of local cohomology is subject to deep analysis throughout the fields. One measure of this vanishing is the cohomological dimension. Let $S$ be a Noetherian local ring and $\mathfrak{a}$ be an ideal of $S$. The cohomological dimension of $\mathfrak{a}$ in $S$ is:

\begin{center}
$\cd (S,\mathfrak{a}) = \sup \{ i \in \mathds{Z} _{\geq 0} | H_\mathfrak{a} ^i (M) \neq 0 $ for some $S$-module $M \}$.
\end{center}

The cohomological dimension has been a topic garnering wide interest (cf. \cite{Ha68, Og73, PS73, HL90, Ly93, Va13, DT16}). It is known that $\cd (S,\mathfrak{a})$ is less than or equal to $\dim S$. We desire to improve this bound for special cases.  We will take advantage of various existing results to aid in this endeavor.  If $\depth S/\mathfrak{a} \geq 1$, then $\cd (S, \mathfrak{a}) \leq n-1$, which is
an immediate consequence of the Hartshorne-Lichtenbaum vanishing theorem (\cite[Theorem 3.1]{Ha68}). From the work of Ogus, we have $\depth S/\mathfrak{a} \geq 2$, then $\cd(S, \mathfrak{a}) \leq n-2$ (Remark below Corollary 2.11 \cite{Og73}).

In \cite{HL90}, Huneke and Lyubeznik prodce a theorem which helps construct more of these bounds. In \cite{DT16}, the theorem is reinterpreted for convenience of application.

\begin{theorem}[{Huneke-Lyubeznik, see \cite[Theorem 3.7]{DT16}}]
Let (S,m) be a $d$-dimensional regular local ring containing a field and let $\mathfrak{a} \subset S$ be an ideal of pure height $c$. Let $f$ $:$ $\mathds{N} \rightarrow \mathds{N}$ be a non-decreasing function. Assume there exist integers $\ell' \geq \ell \geq c$ such that:

($1$) $f(\ell) \geq c$,

($2$) $\cd (S_\mathfrak{p},\mathfrak{a}_\mathfrak{p}) \leq f(\ell'+1)-c+1$ for all prime ideals $\mathfrak{p} \supset \mathfrak{a}$ with $\height \mathfrak{p} \leq \ell-1$,

($3$) $\cd (S_\mathfrak{p},\mathfrak{a}_\mathfrak{p}) \leq f(\height \mathfrak{p})$ for all prime ideals $\mathfrak{p} \supset \mathfrak{a}$ with $\ell \leq \height \mathfrak{p} \leq \ell'$,

($4$) $f(r-s-1) \leq f(r) - s$ for every $r \geq \ell' + 1 $ and every $c-1 \geq s \geq 1$.

Then $\cd (S,\mathfrak{a}) \leq f(d)$ if $d \geq \ell$.
\end{theorem}

In \cite{DT16}, Dao and Takagi proved the following corollary:

\begin{cor}[Dao-Takagi]
Let $(S,\mathfrak{m})$ be an $n$-dimensional regular local ring essentially of finite type over a field. If $\mathfrak{a}$ is an ideal of $S$ such that $\depth S/\mathfrak{a} \geq 3$, then $\cd (S,\mathfrak{a}) \leq n-3$.
\end{cor}

Dao and Takagi then used this corollary in conjuction with Theorem $3.1$ to prove:

\begin{theorem}[Dao-Takagi]
Let $S$ be an $n$-dimensional regular local ring containing a field and let $\mathfrak{a} \subset S$ be an ideal of height $c$.

(1) If $S/\mathfrak{a}$ satisfies Serre's condition $(S_2)$ and $\dim S/\mathfrak{a} \geq 1$, then

\begin{center}
$\cd (S,\mathfrak{a}) \leq n-1-\lfloor \frac{n-2}{c} \rfloor$
\end{center}

(2) Suppose that $S$ is essentially of finite type over a field. If $S/\mathfrak{a}$ satisfies Serre's condition $(S_3)$ and $\dim S/\mathfrak{a} \geq 2$, then

\begin{center}
$\cd (S,\mathfrak{a}) \leq n-2-\lfloor \frac{n-3}{c} \rfloor$
\end{center}

\end{theorem}

In this paper, we generalize this theorem to consider $(S_2^j)$ and $(S_3^j)$ for any $j$.

\begin{theorem}
Let $S$ be an $n$-dimensional regular local ring containing a field and let $\mathfrak{a} \subset S$ be a pure ideal of height $c$.

($1$) If $S/\mathfrak{a}$ satisfies Serre's condition $(S_2^j)$ and $\dim S/\mathfrak{a} \geq 1+j$, then

\begin{center}
$\cd (S,\mathfrak{a}) \leq n-1-\lfloor \frac{n-2-j}{c} \rfloor$
\end{center}

($2$) Suppose that $S$ is essentially of finite type over a field. If $S/\mathfrak{a}$ satisfies Serre's condition $(S_3^j)$ and $\dim S/\mathfrak{a} \geq 2+j$, then

\begin{center}
$\cd (S,\mathfrak{a}) \leq n-2-\lfloor \frac{n-3-j}{c} \rfloor$
\end{center}

\end{theorem}

\begin{proof}
For ($1$) we apply Theorem $3.1$ with $f(m) = m-1-\lfloor \frac{m-2-j}{c} \rfloor$, $\ell=c+1+j$, and $\ell' = 2c+1+j$.  Condition (2) of Theorem $3.1$ becomes $\cd (S_\mathfrak{p}, \mathfrak{a}_\mathfrak{p}) \leq c+j$ for all primes of height at most $c+j$.  This follows since $\cd (S_\mathfrak{p},\mathfrak{a}_\mathfrak{p}) \leq \height \mathfrak{p} \leq c+j$.  Condition (3) of Theorem $3.1$ breaks into two cases.  If $\height \mathfrak{p} = \ell$ then we get $\cd (S_\mathfrak{p},\mathfrak{a}_\mathfrak{p}) \leq \height \mathfrak{p} -1$.  This is true, since $\depth S_\mathfrak{p}/\mathfrak{a}S_\mathfrak{p} \geq \min \{2, c+j+1-c-j\} = 1$.  If $\height \mathfrak{p} > \ell$ then we get $\cd (S_\mathfrak{p},\mathfrak{a}_\mathfrak{p}) \leq \height \mathfrak{p} -2$.  This is true, since $\depth S_\mathfrak{p}/\mathfrak{a}S_\mathfrak{p} \geq 2$.  

For ($2$) we apply Theorem $3.1$ with $f(m) = m-2-\lfloor \frac{m-3-j}{c} \rfloor$, $\ell=c+2+j$, and $\ell' = 2c+2+j$.  Condition (2) of Theorem $3.1$ breaks into two cases.  If $\height \mathfrak{p} = \ell -1$, then $\cd (S_\mathfrak{p},\mathfrak{a}_\mathfrak{p}) \leq \height \mathfrak{p} -1$.  This is true since $\depth S_\mathfrak{p}/\mathfrak{a}S_\mathfrak{p} \geq 1$.  Otherwise, (2) becomes $\cd (S_\mathfrak{p},\mathfrak{a}_\mathfrak{p}) \leq \height \mathfrak{p}$, which is true.  Demonstrating condition ($3$) of Theorem $3.1$ has two parts.  If $\height \mathfrak{p} = \ell = c+2+j$, then $\depth S/\mathfrak{a} \geq 2$, and thus we have $cd (S,\mathfrak{a}) \leq \height \mathfrak{p} -2 = f(\height \mathfrak{p})$. If $c+3+j \leq \height \mathfrak{p} \leq 2c+2+j$, then $f(\height \mathfrak{p}) = \height \mathfrak{p} -3$.  However, $\depth S_\mathfrak{p}/\mathfrak{a}S_\mathfrak{p} \geq 3$.  Therefore, Corollary $3.2$ implies $\cd (S_\mathfrak{p}, \mathfrak{a}_\mathfrak{p}) \leq \height \mathfrak{p} -3$.
\end{proof}


\begin{cor}
Let $S$ be an $n$-dimensional regular local ring containing a field and let $\mathfrak{a} \subset S$ be a pure, square-free monomial ideal of height $c$.

($1$) If $S/\mathfrak{a}$ satisfies Serre's condition $(S_2^j)$ and $\dim S/\mathfrak{a} \geq 1+j$, then

\begin{center}
$\pd S/\mathfrak{a} \leq n-1-\lfloor \frac{n-2-j}{c} \rfloor$
\end{center}

($2$) Suppose that $S$ is essentially of finite type over a field. If $S/\mathfrak{a}$ satisfies Serre's condition $(S_3^j)$ and $\dim S/\mathfrak{a} \geq 2+j$, then

\begin{center}
$\pd S/\mathfrak{a} \leq n-2-\lfloor \frac{n-3-j}{c} \rfloor$
\end{center}

\end{cor}

\begin{proof}

When $\mathfrak{a}$ is a square-free monomial ideal, $\cd (S,\mathfrak{a}) = \pd S/\mathfrak{a}$ \cite{Ly83}. Therefore, applying Theorem $3.4$ we get the desired bounds.
\end{proof}

\section{A generalized Hochster-Huneke graph}

For this section we consider two kinds of rings, local rings and rings which are a quotient of a polynomial ring and a homogeneous ideal.  When we say ring, we shall mean these types of rings unless otherwise noted.

The Hochster-Huneke graph is a graphical representation of the minimal prime ideals of a ring R.  This graph is sometimes called the dual graph of $\Spec R$ and has been examined from many perspectives throughout the literature (cf \cite{HH94,BB15,BV15,Ho16}).  The connectivity of the Hochster-Huneke graph and its localizations provides information about algebraic properties of $R$.  We will consider a generalization of the Hochster-Huneke graph.

 
\begin{defn}[\cite{HH94}]
Let $G (R)$ be the graph with $V(G (R)) = \{ v_i = P_i \}$ where the $P_i$ are the minimal primes of $R$, $E(G (R)) = \{ (v_k, v_\ell) | \height (\mathfrak{p}_k + \mathfrak{p}_\ell) = 1 \}$. Then $G(R)$ is the $\textit{Hochster-Huneke}$ $\textit{graph}$ of $R$.
\end{defn}

\begin{defn}
A ring is \textit{locally connected} if for any $\mathfrak{p} \in \Spec R$, $G(R_\mathfrak{p})$ is connected.
\end{defn}

\begin{theorem}[{\cite[Theorem 6.1.5]{Ku08}, \cite[Theorem 2.15]{Ho16}}]
Let $R$ be a Stanley-Reisner ring. Then $R$ satisfies $(S_2)$ if and only if $R$ is locally connected.
\end{theorem}

We desire to make a generalized Hochster-Huneke graph in order to make an analogue of the previous theorem for rings satisfying $(S_2^j)$.  We note that our generalization is the same generalization given in \cite{NS17}.

\begin{defn}
Let $G^j(R)$ be the graph with $V(G^j(R)) = \{v_i = P_i \}$ where the $\mathfrak{p}_i$ are the minimal primes of $R$, $E(G^j(R)) = \{ (v_k, v_\ell) | 1 \leq \height (\mathfrak{p}_k + \mathfrak{p}_\ell) \leq j \}$. We note $G^1(R)$ is the Hochster-Huneke graph of $R$.
\end{defn}


\begin{defn}
A ring is \textit{$j$-locally connected} if for any $\mathfrak{p} \in \Spec R$, $G^{j}(R_\mathfrak{p})$ is connected.
\end{defn}

\begin{theorem}
Let $R$ be a Stanley-Reisner ring. Then, $R$ satisfies $(S_2^j)$ if and only if $R$ is $j+1$-locally connected.
\end{theorem}

\begin{proof}
By Theorem $2.1$, we have $R$ satisfying $(S_2^j)$ can be viewed as a condition on $\Supp \Ext _S^{n-i}(R,S)$. From [Ya00], we have that $R$ is a Stanley-Reisner ring implies that $\Ext _S^{n-i}(R, S)$ is a square-free module. Thus, $\Ext _S^{n-i}(R, S)$ is uniquely determined by its prime ideals generated by variables. Therefore, we only need to consider primes generated by variables when showing $R$ satisfies $(S_2^j)$.

Assume $R$ is $(j+1)$-locally connected and $\dim R = d$. If $j \geq d-1$, then $1 \geq d-j$.  All Stanley-Reisner rings are $(S_1)$, thus this case is trivially true.

Now we suppose $j < d-1$. When $d > j+1$, $G^{j+1}(R)$ being connected implies $\Delta_R$ is connected. Thus, for all localizations of $R$ such that $\dim R_\mathfrak{p} > j+1$ we have that $\Delta_{R_{\mathfrak{p}}}$ is connected. Therefore, we have that $\depth R_\mathfrak{p} \geq 2$, for all $\mathfrak{p} \in \Spec R$ with $\dim R_\mathfrak{p} \geq 2+j$. This combined with the fact that all Stanley-Reisner rings satisfy $(S_1)$ gives that $R$ satisfies $(S_2^{j})$.

Now suppose $R$ satisfies $(S_2^{j})$. By Corollary 2.3 of \cite{Ha62}, we have that $R$ is locally connected in codimension $1+j$.  Thus $G^{j+1}(R_\mathfrak{p})$ is connected for all $\mathfrak{p} \in \Spec R$.

\end{proof}

\section{Resolution of the Alexander Dual}

Let us use notation from \cite{FM14} to introduce Alexander Duality. Let $\mathbb{K}$ be a field and $S=\mathbb{K}[x_1,...,x_n]$.  Let $m$ be a monomial. Let $\mathfrak{p}_m = (x_i : x_i | m)$.

A prime ideal $\mathfrak{p}$ is an associated prime of a square free monomial ideal $I$ if $\mathfrak{p} \supseteq I$ and $\mathfrak{p}$ is minimal among primes ideals that contain $I$. The set of all associated primes of $I$ is written $\Ass (I)$.

\begin{defn}
Let $I$ be a squarefree monomial ideal. The \textit{Alexander dual} of $I$ is \[ I^{\vee} = (m : \mathfrak{p}_m \in \Ass(I)). \]
\end{defn}

In \cite{ER98}, Eagon and Reiner proved that a Stanley-Reisner ring $R/I$ is Cohen-Macaulay if and only if $I^\vee$ has linear resolution.  In \cite{Ya00}, Yanagawa generalized their result to say a Stanley-Reisner ring $R=S/I$ satisfies $(S_{\ell})$ if and only if $I^\vee$ has syzygy matrices with linear entries up to homological degree $\ell-1$ \cite{Ya00}. We further generalize these statements for equidimensional Stanley-Reisner rings satisfying $(S_{\ell} ^j)$.

\begin{defn}
Let $R = S/I$ be an equidimensional Stanley-Reisner ring.  We say that $I^{\vee}$ satisfies $(N_{c,\ell}^j)$ if \[ [ \Tor_{\gamma}(I^\vee,\mathbb{K})]_{\beta} = 0 \textrm{ for all } \gamma < \ell \textrm{ and for all } c+j+\gamma < \beta \leq n. \]
\end{defn}
\begin{theorem}
Let $R=S/I$ be a $d$-dimensional, equidimensional Stanley-Reisner ring with codimension $c$. Then the following are equivalent for $\ell \geq 2$:
\begin{enumerate}[(i)]
\item $R$ satisfies $(S_{\ell}^j)$.
\item $I^\vee$ satisfies $(N_{c,\ell}^j)$.
\end{enumerate}
\end{theorem}

\begin{proof}
By Corollary $2.2$, $R$ satisfies $(S_{\ell}^j)$ if and only if
\[ \dim \Ext_S^{n-i}(R,S) \leq i-\ell \textrm{ for all } i = 0,..., d-j-1. \]
We rewrite for convenience:
\[ \dim \Ext_S^{\alpha}(R,S) \leq n-\alpha-\ell \textrm{ for all } c+j < \alpha \leq n.\]
The above is true if and only if for any prime $\mathfrak{p}$ with height $h < \alpha + \ell$ and support $F$ we have:
\[ (\Ext_S^{\alpha}(R,S))_F = 0 \textrm{ for } c+j < \alpha \leq n. \]
By Theorem $3.4$ of [Ya00], this is true if and only if
\[ [\Tor_{h-\alpha}(I^\vee,\mathbb{K})]_{F^c}=0 \textrm{ for } c+j < \alpha \leq n. \]
By rewriting, we recover our definition of $(N_{c,\ell}^j)$.

\end{proof}

\begin{remark}
From this result, we have $R$ satisfies $(S_2^j)$ if and only if all entries of the first syzygy matrix of $R$ have degree at most $j+1$.  Furthermore, $R$ satisfies $(S_\ell^j)$ if the sum of the highest degrees of the first $\ell -1$ syzygy matrices is less than $j+\ell -1$.
\end{remark}

\begin{cor}
Let $\mathfrak{a}$ be a pure, square-free monomial ideal of an $n$-dimensional polynomial ring over a field.

($1$) If $\mathfrak{a}$ satisfies $(N_{c,2}^j)$ and $c \leq n-1$, then

\begin{center}
$\reg \mathfrak{a} \leq n-1-\lfloor \frac{n-2-j}{c} \rfloor$
\end{center}

($2$) If $\mathfrak{a}$ satisfies $(N_{c,3}^j)$ and $c \leq n-2$, then
\begin{center}
$\reg \mathfrak{a} \leq n-2-\lfloor \frac{n-3-j}{c} \rfloor$
\end{center}

\end{cor}

\begin{proof}

From Theorem 5.3 we get that $S/\mathfrak{a}^\vee$ satisfies $(S^j_2)$, and $\cd (S,\mathfrak{a}^\vee) = \reg \mathfrak{a}$. Therefore, applying Theorem $3.4$ we get the desired bounds.
\end{proof}

\section{A generalization of Reisner's criterion}

Reisner's criterion is a method for checking the Cohen-Macaulayness of a Stanley-Reisner ring by considering the reduced homology groups of the ring's associated complex. The $i$th reduced homology group of $\Delta$ is $\tilde{H}_i(\Delta ; \mathbb{K})$. The following theorem is from \cite{Re76}.

\begin{theorem}[Reisner's Criterion]
Let $\mathbb{K}$ be a field and $\Delta$ be a simplicial complex of dimension $d - 1$. Then $\Delta$ is Cohen-Macaulay over
$\mathbb{K}$ if and only if for every $F \in \Delta$ and for every $i < \dim(\lk_{\Delta}(F))$, $\tilde{H}_i(\lk_{\Delta}(F);\mathbb{K}) = 0$ holds true.
\end{theorem}

Terai has formulated an analogue of this theorem for Stanley-Reisner rings satisfying $(S_\ell)$ ($\ell \geq 2$). 

\begin{theorem}[\cite{Te07} described after Theorem $1.4$]
Let $\mathbb{K}$ be a field and $\Delta$ be a simplicial complex of dimension $d - 1$. Then $\Delta$ satisfies $(S_\ell)$ over
$\mathbb{K}$ if and only if for every $F \in \Delta$ (including $F = \emptyset$) with $|F| \leq d-i-2$ and for every $-1\leq i \leq \ell-2$, $\tilde{H}_i(\lk_{\Delta}(F);\mathbb{K}) = 0$ holds true.
\end{theorem}

Now we present a lemma, which serves as an analogue to one direction of Reisner's criterion for rings satisfying $(S_\ell^j)$.  Following this Lemma, we prove a full analogue of Reisner's criterion for equidimensional rings satisfying $(S_\ell^j)$.  For the rest of this section, $R$ shall be the Stanley-Reisner ring of $\Delta$ and $\mathfrak{m}$ shall be the unique homogeneous maximal ideal of $R$.

\begin{lemma}
Let $\mathbb{K}$ be a field and $\Delta$ be a simplicial complex of dimension $d - 1$.  For a face $F$ let $d_F$ denote the cardinality of the largest facet containing $F$. Then $\Delta$ satisfying $(S_\ell^j)$ ($\ell \geq 2$) over
$\mathbb{K}$ implies for every $F \in \Delta$ (including $F = \emptyset$) with $|F| \leq d_F-i-j-2$ and for every $-1\leq i \leq \ell-2$, $\tilde{H}_i(\lk_{\Delta}(F);\mathbb{K}) = 0$.

\end{lemma}

\begin{proof}
Consider $\tilde{H}_{i}(\lk_{\Delta}(F);\mathbb{K})$ with $-1 \leq i \leq \ell-2$ and $|F| \leq d_F-i-j-2$.  Let us localize $R$ at a prime $\mathfrak{p}$ generated by the variables contained in $F$.  Let $\dim R_\mathfrak{p} = d_\mathfrak{p}$.  Then $(S_\ell^j)$ condition implies $H_{\mathfrak{p}R_\mathfrak{p}}^{\beta} (R_\mathfrak{p}) = 0$ for all $\beta < \min (\ell, d_\mathfrak{p}-j):=b_\mathfrak{p}$.  
Hochster's formula, then, implies $\tilde{H}_{\beta - |F'|-1}(\lk_{\lk F}(F');\mathbb{K}) = 0$ for all $F'$ and all $\beta < b_\mathfrak{p}$.  In particular, $\tilde{H}_{\beta -1}(\lk_{\lk F}(\emptyset);\mathbb{K}) = \tilde{H}_{\beta-1}(\lk_{\Delta}(F);\mathbb{K}) = 0$ for all $\beta < b_\mathfrak{p}$.  If we can show $i+1 < b_\mathfrak{p}$, then we will have $\tilde{H}_{i}(\lk_{\Delta}(F);\mathbb{K}) = 0$ as desired.  We have $i \leq \ell-2$ which implies $i+1 < \ell$.  We also have $|F| \leq d_F-j-i-2$, which implies that $i < d_F-j-1-|F|$.  Thus, $i+1 < d_F-|F|-j$, and $d_F-|F| = d_\mathfrak{p}$.  Thus, $i+1 < b_\mathfrak{p}$.
\end{proof}


\begin{theorem}
Let $\mathbb{K}$ be a field and $\Delta$ be a pure simplicial complex of dimension $d - 1$. Then $\Delta$ satisfies $(S_\ell^j)$ over
$\mathbb{K}$ if and only if for every $F \in \Delta$ (including $F = \emptyset$) with $|F| \leq d-i-j-2$ and for every $-1\leq i \leq \ell-2$, $\tilde{H}_i(\lk_{\Delta}(F);\mathbb{K}) = 0$ holds true.
\end{theorem}

\begin{proof}
Suppose $\Delta$ satisfies $(S_\ell^j)$.  Then the result follows from lemma $6.3$.

Suppose for every $F \in \Delta$ with $|F| \leq d-i-j-2$ and for every $-1\leq i \leq \ell-2$, $\tilde{H}_i(\lk_{\Delta}(F);\mathbb{K}) = 0$.  

Then let us consider $\tilde{H}_{\alpha - |F| -1}(\lk_{\Delta}(F);\mathbb{K})$.  Let us examine the case where $\alpha < b:=\min \{\ell, d-j\}$.  We have $|F| > d-i-j-2=d-\alpha+|F|+1-j-2=d-\alpha-1+|F|-j$ if and only if $\alpha+1>d-j$ which implies $\alpha \geq d-j \geq b$.  Thus, $\tilde{H}_{\alpha - |F| - 1}(\lk_{\Delta}(F);\mathbb{K}) = 0$ for any $F$ so long as $\alpha < b$.  We also note that $\alpha < b$ implies $\alpha <\ell$ which implies $\alpha -|F|-1 \leq \ell-2$.  Thus, if we have $\alpha < b$, we have $\tilde{H}_{\alpha - |F| -1}(\lk_{\Delta}(F);\mathbb{K}) = 0$ for all $F$.  Applying Hochster's formula, we get that $H_\mathfrak{m}^i(R)=0$ for all $i < b$.  Therefore, we have $\depth R \geq b$.

We must prove that for any localization $R_\mathfrak{p}$ with dimension $d_\mathfrak{p}$, we have $\depth R_\mathfrak{p} \geq b_\mathfrak{p}:= \min \{ \ell,d_\mathfrak{p}-j \}$.  From~\cite{Ya00}, we have that $R$ is a Stanley-Reisner ring implies that $\Ext _S^{n-i}(R,S)$ is a square-free module.  Thus, $\Ext _S^{n-i}(R,S)$ is uniquely determined by its primes generated by variables.  Support of $\Ext _S^{n-i}(R,S)$ determines if $R$ satisfies $(S_\ell^j)$.  Therefore, we only need consider primes generated by variables when showing $R$ satisfies $(S_\ell^j)$.  To examine $R_\mathfrak{p}$ we consider the simplicial complex $\lk_\Delta F$, where $F$ is generated by the variables which are not generators of $R$ (see Remark $1.4$).  

We then have $\tilde{H}_{i}(\lk_{\lk F}(F');\mathbb{K}) = \tilde{H}_{i}(\lk_{\Delta}(F \cup F');\mathbb{K})$ which is $0$ when $|F \cup F'| \leq d-i-j-2$ and when $i \leq \ell-2$.  Thus, we have $\tilde{H}_{i}(\lk_{\lk F}(F');\mathbb{K}) = 0$ when $|F|+|F'| \leq d-i-j-2$ and $i \leq \ell-2$, and thus $\tilde{H}_{i}(\lk_{\lk F}(F');\mathbb{K}) =0$ when $|F'| \leq d-|F|-i-j-2$.  We have $d-|F|\geq d_\mathfrak{p}$.  Thus, $\tilde{H}_{i}(\lk_{\lk F}(F');\mathbb{K}) =0$ when $|F'| \leq d_\mathfrak{p}-i-j-2$ and $i \leq \ell-2$.  Thus, by the above argument, we have $\depth R_\mathfrak{p} \geq \min \{ \ell, d_\mathfrak{p}-j \}$.  Thus, we have $R$ satisfies $(S_\ell^j)$.  
\end{proof}

\begin{cor}
Let $\mathbb{K}$ be a field and $\Delta$ be a pure simplicial complex. Then the Stanley-Reisner ring of $\Delta$, $R$, satisfies $(S_2^j)$ over $\mathbb{K}$ if and only if for every face $F \in \Delta$ with $\dim (\lk _{\Delta} (F)) \geq 1+j$, $\lk _{\Delta} (F)$ is connected. Note $(S_2^j)$ is independent of the base field.
\end{cor}

\begin{proof}
Suppose $\Delta$ satisfies $(S_2^j)$.  Then for any prime ideal $\mathfrak{p}$ with $\dim R_\mathfrak{p} \geq 2+j$ we have $\depth R_\mathfrak{p} \geq \min \{ 2, \dim R_\mathfrak{p}-j \} = 2$.  Now we take any face $F$ of our complex such that $\dim \lk_{\Delta} F \geq 1+j$.  The Stanley-Reisner ring of $\lk_{\Delta} F$ is $R_{\tilde{\mathfrak{p}}}$ where $\tilde{\mathfrak{p}}$ is the prime ideal generated by the variables not contained in $F$.  Thus, $\dim R_{\tilde{\mathfrak{p}}} \geq 2+j$.  Thus, $H_m^1(R_{\tilde{\mathfrak{p}}})=0$.  Thus, by Hochster's formula, we have $\tilde{H}_{1-|F'|-1}(\lk_{\lk F} F') = 0$ for all $F'$.  In particular, $\tilde{H}_{0}(\lk_{\lk F} \emptyset) = 0$, which implies $\lk_{\lk F} \emptyset = \lk_{\Delta}F$ is connected.  

Now suppose $\dim \lk F \geq 1+j$ implies $\lk F$ is connected.  Thus, we have $\tilde{H}_0(\lk_{\lk F} \emptyset) = 0$.  The link of $F$ has dimension at least $1$.  Thus, if we take a face of $\lk F$ $F'$ such that $|F'|=1$, then $\lk _{\lk F} F'$ is not empty.  Thus, $\tilde {H}_{-1}(\lk_{\lk F} F') = 0$ for all $|F'| = 1$.  Thus, we get that $H_\mathfrak{m}^1(R_\mathfrak{p}) = 0$ where $\mathfrak{p}$ is the prime generated by the variables contained in $F$.  Thus, we have $\depth R_\mathfrak{p} \geq 2$.  This argument holds for all primes generated by variables such that $\dim R_\mathfrak{p} \geq 2+j$.  If we have a prime $\mathfrak{p}$ such that $\dim R_\mathfrak{p} < 2+j$, we still have $\depth R_\mathfrak{p} \geq 1$, since $R_\mathfrak{p}$ is a Stanley-Reisner ring.  Therefore, $R$ satisfies $(S_2^j)$.  
\end{proof}

\section{Monomial Ideals}

In this section, we consider monomial ideals that are not necessarily squarefree.  For convenience, we will establish a few conventions.  When we consider the depth of a Stanley-Reisner ring, we shall be considering its depth with respect to the unique homogeneous maximal ideal.  When we speak of the localized Stanley-Reisner ring of a link, we shall mean the Stanley-Reisner ring of the link localized at its unique homogeneous maximal ideal.  Throughout this section, let $\mathbb{K}$ be a field.

One of the most powerful techniques for working with such ideals is polarization.  We reproduce the definition of \cite{MT09}.

\begin{defn}
Let $I$ be a monomial ideal in $\mathbb{K}[x_1,...,x_n]$. Let the minimal generators of $I$ be $u_1,...,u_m$ where $u_i = \prod_{j=1}^n x_j^{a_{ij}}$. For $1 \leq j \leq n$, let $a_j = \max \{ a_{ij} \}$, and let $N=\max \{ a_j \}$. Then let \[ T = \mathbb{K}[x_{1,1}x_{1,2},...,x_{1,N},x_{2,1},x_{2,2},...,x_{2,N},...,x_{n,1},x_{n,2},...,x_{n,N}]. \]

\end{defn}

\begin{defn}
The polarization of a monomial $u = x_1^{a_1}x_2^{a_2}...x_n^{a_n}$ is 

\[ \pol(u) = \prod_{1 \leq k \leq n, a_k \neq 0} (x_{k,1}x_{k,2}...x_{k,a_k})\]

\end{defn}

Let $J$ be the squarefree monomial ideal of $T$ generated by $\{ \pol(u_i) \}$.  We call $J$ the \textit{polarization} of $I$. 

In \cite{HT05}, Herzog, Takayama, and Terai proved that for a monomial ideal $I$, $S/I$ being Cohen-Macaulay implies $S/\sqrt{I}$ is Cohen-Macaulay.  In \cite{PF14}, Pournaki et al. generalize this result to say $S/I$ satisfying $(S_\ell)$ implies $S/\sqrt{I}$ satisfies $(S_\ell)$.  We generalize this result further to hold for $(S_\ell^j)$.

\begin{lemma}
Localization preserves $(S_\ell^j)$.
\end{lemma}

\begin{proof}
Let $R$ satisfy $(S_\ell^j)$.  Let $S$ be a multiplicatively closed set in $R$.  Let $\mathfrak{q}$ be a prime ideal of $R_S$.  The ideal $\mathfrak{q}$ is the extension of a prime ideal $\mathfrak{p}$ in $R$ and so $(R_S)_\mathfrak{q} \cong R_\mathfrak{p}$.  Since $R$ satisfies $(S_\ell^j)$, \[\depth (R_S)_\mathfrak{q} = \depth R_\mathfrak{p} \geq \min\{ \ell, \dim R_\mathfrak{p} -j \} =\min \{ \ell, \dim (R_S)_\mathfrak{q} -j \}.\]  Therefore, $R_S$ satisfies $(S_\ell^j)$.
\end{proof}

In \cite{MT09}, Murai and Terai prove that polarization preserves $(S_\ell)$.  We will generalize their argument in the following lemmas.

\begin{lemma}
If $\depth R \geq \dim R - j$ then $\depth R_\mathfrak{p} \geq \dim R_\mathfrak{p} - j$ for all $\mathfrak{p} \in \Spec R$ where $\mathfrak{p}$ is generated by variables.
\end{lemma}

\begin{proof}
The ring $R_\mathfrak{p}$ where $\mathfrak{p}$ is a prime generated by variables is the $0$ ring, a field, or the localized Stanley-Reisner ring of the link of F, which is generated by the variables not in the generating set of $\mathfrak{p}$.  From \cite{MT09}, we have $\depth(\mathbb{K}[\lk_\Delta(v)]) \geq \depth(\mathbb{K}[\Delta])-1$.  Therefore, $\depth R_\mathfrak{p} \geq \depth R - |F| \geq \dim R - j -|F| \geq \dim R_\mathfrak{p} -j$.
\end{proof}

\begin{lemma}
Polarization preserves $(S_\ell^j)$.
\end{lemma}

\begin{proof}

We mimic the proof of \cite{MT09}.

Let us consider an $(S_\ell^j)$ ring $\mathbb{K}[x_1,...,x_n]/I$.  Let $I^{pol} \subseteq T$ be the polarization of $I$.  Let $\Delta$ be the complex whose Stanley-Reisner ring is $T/I^{pol}$.  Let $F$ be an arbitrary face of $\Delta$.  We will be finished when we prove $\depth (\mathbb{K}[ \lk_{\Delta}(F)]) \geq \min \{ r, \dim \lk_{\Delta} (F) +1-j \}$.

Let us write $V_k = \{ x_{k,1}, x_{k,2},...,x_{k,N} \}$.  We may write $F = F_1 \cup F_2 \cup \cdots \cup F_s \cup V_{s+1} \cup \cdots \cup V_n$, where $F_k$ is a proper subset of $V_k$ for $1 \leq k \leq s$.

Set $S' = \mathbb{K}[x_1,...,x_s]$, $I' = (I:x_{s+1}^N,...,x_n^N) \cap S'$.  Let $\mathfrak{p}= \langle x_1,...,x_s \rangle$.  Then:

\[ \depth(S'/I') = \depth(S/I)_\mathfrak{p} \geq \min \{r, \dim (S/I)_\mathfrak{p}-j \} = \min \{r, \dim (S'/I')-j \} \]

Let $B = \mathbb{K}[\cup_{k=1}^{s} V_k]$.  Let $J \subseteq B$ be the monomial ideal generated by the polarization of the minimal generators of $I'$.  Let $\Gamma$ be the set of all squarefree monomials in $B$ which are not in $J$.  We note $\Gamma$ is a simplicial complex if we identify squarefree monomials with sets of variables.  Since $J$ is a squarefree monomial ideal, $J$ is generated by the monomials that generate $I_{\Gamma}$ and the variables of $B$ which are not in $\Gamma$.  Therefore, we have $\depth (\mathbb{K}[\Gamma]) = \depth(B/J)$.  Because $J$ is  the polarization of $I'$ we get (from \cite{BH98}):

\[ \depth ( \mathbb{K}[\Gamma]) = \depth(B/J)=\depth(S'/I') +s(N-1). \]

By construction of $\Delta$ and $\Gamma$, we get $\lk_{\Delta}(V_{s+1} \cup ... \cup V_n) = \Gamma$.  Thus, $\lk_\Delta(F)=\lk_\Gamma(F_1 \cup .... \cup F_s)$.

Thus, it will be sufficient to show \[ \depth (\mathbb{K}[\lk_\Gamma(F_1 \cup \cdots \cup F_s)]) \geq \min \{ r, \dim \lk_\Gamma(F_1 \cup \cdots \cup F_s) +1-j \}. \]

Suppose $\depth (S'/I') \geq \dim (S'/I') - j$.  


From the process of polarization we can see that \[ \dim(B/J) = \dim(S'/I') + s(N-1). \]

Thus, we have $\depth B/J \geq \dim B/J - j.$  By lemma $7.4$, we have \[ \depth (\mathbb{K}[\lk_\Gamma(F_1 \cup \cdots \cup F_s)]) \geq \dim \lk_\Gamma(F_1 \cup \cdots \cup F_s) +1-j. \]

Now suppose $\depth (S'/I') \geq \ell$.  From \cite{MT09}, we have \[ \depth (\mathbb{K}[\lk_\Gamma(F_1 \cup \cdots \cup F_s)]) \geq \ell. \]

\end{proof}

\begin{theorem}
Let $A=\mathbb{K}[x_1, \cdots , x_n]$ be the polynomial ring in $n$ variables over a field $\mathbb{K}$.  Let $I$ be a monomial ideal of $A$.  Suppose that $A/I$ satisfies $(S_{\ell}^j)$. Then $A/\sqrt(I)$ satisfies $(S_{\ell}^j)$.
\end{theorem}

\begin{proof}
 Using the lemmas and theorem in this section, we follow the argument of Theorem 8.3 of Pournaki et al \cite{PF14}.
 
Let $T/J$ be the polarization of $A/I$, where $T=\mathbb{K}[x_1,...,x_n,Y]$ is the new polynomial ring over $\mathbb{K}$ with Y as the set of new variables.  By Theorem $7.5$, $T/J$ satisfies $(S_{\ell}^j)$.  Let $W$ be the multiplicatively closed subset $\mathbb{K}[Y] \backslash 0$ in $T$ and consider $F=\mathbb{K}(Y)$.  Then the localization of $T/J$ at W is isomorphic to:
 
 \[ F[x_1,...,x_n]/\sqrt(I) \cong (\mathbb{K}[x_1,...,x_n]/\sqrt(I)) \otimes_\mathbb{K} F. \]
 
Since localization preserves generalized Serre's condition, $(\mathbb{K}[x_1,...,x_n]/\sqrt(I)) \otimes_\mathbb{K} F$ satisfies $(S_\ell^j)$.  Now to finish the proof, we will mirror the argument found in Section 2.1 of \cite{BH98}.  We wish to show $(\mathbb{K}[x_1,...,x_n]/\sqrt(I)) \otimes_\mathbb{K} F $ satisfying $(S_\ell^j)$ implies $(\mathbb{K}[x_1,...,x_n]/\sqrt(I))$ satisfies $(S_\ell^j)$.  Let us consider $\mathfrak{p} \in \Spec (\mathbb{K}[x_1,...,x_n]/\sqrt(I))$.  Since $(\mathbb{K}[x_1,...,x_n]/\sqrt(I)) \otimes_\mathbb{K} F$ is faithfully flat, there exists a $\mathfrak{q} \in \Spec (\mathbb{K}[x_1,...,x_n]/\sqrt(I)) \otimes_\mathbb{K} F$ such that $\mathfrak{p} = (\mathbb{K}[x_1,...,x_n]/\sqrt(I)) \cap \mathfrak{q}$ and $\mathfrak{q}$ is minimal over $\mathfrak{p}((\mathbb{K}[x_1,...,x_n]/\sqrt(I)) \otimes_\mathbb{K} F)$.  Therefore, $\height \mathfrak{p} = \height \mathfrak{q}$.  

We will take $M=R=(\mathbb{K}[x_1,...,x_n]/\sqrt(I))_\mathfrak{p}$, $N=S = ((\mathbb{K}[x_1,...,x_n]/\sqrt(I)) \otimes_\mathbb{K} F)_\mathfrak{q}$.  Since $(\mathbb{K}[x_1,...,x_n]/\sqrt(I)) \otimes_\mathbb{K} F$ is flat over $(\mathbb{K}[x_1,...,x_n]/\sqrt(I))$, N is flat over R.
 Thus, applying Theorem 1.2.16 and A.11 from \cite{BH98} we get: 
 
\[ \depth_S M \otimes_R N = \depth_R M + \depth_S N/mN\]

\[ \dim_S M \otimes_R N = \dim_R M + \dim_S N/mN\]

The fiber of the extension $R \rightarrow S$ is a localization of $\mathbb{K}(\mathfrak{p}) \otimes F$.  Thus, the fiber is Cohen-Macaulay by Proposition 2.11 of \cite{BH98}.  The fiber of that map is also $S_\mathfrak{q}/\mathfrak{p}S_\mathfrak{q} = N_\mathfrak{q}/\mathfrak{p}N_\mathfrak{q}$, because $\mathfrak{p}$ is the maximal ideal of $R_\mathfrak{p}$.  Thus, $N_\mathfrak{q}/\mathfrak{p}N_\mathfrak{q}$ is Cohen-Macaulay.  Thus, subtracting the above equations gives:
\[ \dim M \otimes_R N - \depth_S M \otimes_R N = \dim M - \depth_R M. \]

Since $\height \mathfrak{p} = \height \mathfrak{q}$, we get $\dim R_\mathfrak{p} = \dim S_\mathfrak{q}$.  Thus, $\depth R_\mathfrak{p} = \depth S_\mathfrak{q} \geq \min \{ r, \dim S_\mathfrak{q}-j\} = \min \{ r, \dim R_\mathfrak{p}-j\}$.

\end{proof}

\section{Skeletons}
In this section, we examine the $i$-skeletons of simplicial complexes.  We begin with the definition.

\begin{defn}
The simplicial complex $\Delta ^{(i)} := \{ F \in \Delta | \dim F \leq  i \}$ is the \textit{i-skeleton} of $\Delta$.
\end{defn}

These $i$-skeletons are smaller complexes that retain many important properties of the original complex.  In particular, these complexes are a powerful tool for understanding depth.  A complex whose Stanley-Reisner ring has depth $b$ has a Cohen-Macaulay $b-1$-skeleton.

\begin{prop}[{\cite[Proposition $2.3$]{HT11}}]
Let $\mathbb{K}$ be a field and $\Delta$ a simplicial complex.  If $\Delta$ satisfies $(S_\ell)$ over $\mathbb{K}$, then $\Delta ^{(i)}$ also satisfies $(S_\ell)$ over $\mathbb{K}$ for $(2  \leq \ell \leq i+1)$.  
\end{prop}

\begin{theorem}
Let $\mathbb{K}$ be a field and $\Delta$ a simplicial complex with Stanley-Reisner ring $R$.  If $\Delta$ is pure and $R$ satisfies $(S_\ell^j)$ over $\mathbb{K}$, then $\Delta ^{(i)}$ is also pure and its Stanley-Reisner ring satisfies $(S_\ell^{\max \{ 0, j+i+1-d\}})$ over $\mathbb{K}$ for $(2  \leq \ell \leq i+1)$.  
\end{theorem}

\begin{proof}

Let us consider $R$ that satisfies $(S_\ell^j)$.  Then $\depth R _\mathfrak{p} \geq \min \{ \ell, \dim R_\mathfrak{p} -j \}$ for all $\mathfrak{p} \in \Spec R$.  As noted earlier, we may restrict our consideration to primes generated by variables.  Let $\Delta_\mathfrak{p}$ be the simplicial complex whose localized Stanley-Reisner ring is $R_\mathfrak{p}$.  We note $(\Delta_\mathfrak{p})^{(i-d+\dim R_\mathfrak{p})}$ $\cong (\Delta ^{(i)})_\mathfrak{p}$ when $d-\dim R_\mathfrak{p} \leq i$.  We also note that if $d-\dim R_\mathfrak{p} > i$ then $(\Delta ^{(i)})_\mathfrak{p}$ has Stanley-Reisner ring $\mathbb{K}$ or $0$, and thus is Cohen-Macaulay.  It is known that the $b$-skeleton of a simplicial complex is Cohen-Macaulay if and only if the depth of that complex is greater than or equal to $b+1$.

We consider two cases.  First, $i +j+1 \leq d$.  This is equivalent to $|(\Delta_\mathfrak{p})^{(i-d+\dim R_\mathfrak{p})}| \leq \dim R_\mathfrak{p} -j$.  If $\depth R_\mathfrak{p} \geq i-d+\dim R_\mathfrak{p}+1$ then $(\Delta_\mathfrak{p})^{(i-d+\dim R_\mathfrak{p})}$ is Cohen-Macaulay.  Otherwise, $\depth R_\mathfrak{p} < |(\Delta_\mathfrak{p})^{(i-d+\dim R_\mathfrak{p})}| \leq \dim R_\mathfrak{p} -j$.  Thus, $\depth R_\mathfrak{p} \geq \ell$.  Thus, $\depth (\Delta_\mathfrak{p})^{(i-d+\dim R_\mathfrak{p})} \geq \ell$.  Combining these, we get:

$\depth (\Delta_\mathfrak{p})^{(i-d+\dim R_\mathfrak{p})} \geq \min \{ \ell, |(\Delta_\mathfrak{p})^{(i-d+\dim R_\mathfrak{p})}| \}$.  

Therefore, the Stanley-Reisner ring of $(\Delta)^{i}$ satisfies $(S_\ell)$.

Now consider the case $i +j +1 > d$.  Again if $\depth R_\mathfrak{p} \geq i-d+\dim R_\mathfrak{p}+1$ then $(\Delta_\mathfrak{p})^{(i-d+ dim R_\mathfrak{p})}$ is Cohen-Macaulay.  Otherwise, $\depth (\Delta_\mathfrak{p})^{(i-d+ \dim R_\mathfrak{p})} =\depth R_\mathfrak{p} \geq \min \{ \ell, \dim R_\mathfrak{p} - j \} = \min \{ \ell, i+1-d + \dim R_\mathfrak{p} - (i+j+1-d)\}$.  Thus, we have the Stanley-Reisner ring of $(\Delta ^{(i)})_\mathfrak{p}$ satisfies $(S_\ell^{i+j+1-d})$.  Combining these two cases gives the proof of the theorem.

\end{proof}

\begin{theorem}
Let $\mathbb{K}$ be a field and $\Delta$ a simplicial complex with Stanley-Reisner ring $R$.  If $R$ satisfies $(S_\ell^j)$ over $\mathbb{K}$, then the Stanley-Reisner ring of $\Delta ^{(i)}$ satisfies $(S_\ell^j)$ over $\mathbb{K}$ for $(2  \leq \ell \leq i+1)$.  
\end{theorem}

\begin{proof}
This proof will be similar to the previous proof.

Let us consider $R$ that satisfies $(S_\ell^j)$.  

We consider two cases.  First, $i+1-n+\dim S_\mathfrak{p} \leq \dim R_\mathfrak{p} -j$.  This is equivalent to 

$|(\Delta_\mathfrak{p})^{(i-n+\dim S_\mathfrak{p})}| \leq \dim R_\mathfrak{p} -j$.  If $\depth R_\mathfrak{p} \geq i-n+\dim S_\mathfrak{p}+1$ then $(\Delta_\mathfrak{p})^{(i-n+\dim S_\mathfrak{p})}$ is Cohen-Macaulay.  Otherwise, $\depth R_\mathfrak{p} < |(\Delta_\mathfrak{p})^{(i-n+\dim S_\mathfrak{p})}| \leq \dim R_\mathfrak{p} -j$.  Thus, $\depth R_\mathfrak{p} \geq \ell$.  Thus, $\depth (\Delta_\mathfrak{p})^{(i-n+\dim S_\mathfrak{p})} \geq \ell$.  Combining these, we get

$\depth (\Delta_\mathfrak{p})^{(i-n+\dim S_\mathfrak{p})} \geq \min \{ \ell, |(\Delta_\mathfrak{p})^{(i-n+\dim S_\mathfrak{p})}| \}$.  

Therefore, the Stanley-Reisner ring of $(\Delta)^{i}$ satisfies $(S_\ell)$.

Now consider the case $i +1-n+\dim S_\mathfrak{p} > \dim R_\mathfrak{p}-j$.  Again if $\depth R_\mathfrak{p} \geq i-n+\dim S_\mathfrak{p}+1$ then $(\Delta_\mathfrak{p})^{(i-n+ dim S_\mathfrak{p})}$ is Cohen-Macaulay.  Otherwise, $\depth (\Delta_\mathfrak{p})^{(i-n+ \dim S_\mathfrak{p})} =\depth R_\mathfrak{p} \geq \min \{ \ell, \dim R_\mathfrak{p} - j \} = \min \{ \ell, i+1-n + \dim S_\mathfrak{p} - (i+j+1-n-\dim R_\mathfrak{p} + \dim S_\mathfrak{p})\}$.  We note if $i+j+1-n-\dim R_\mathfrak{p} + \dim S_\mathfrak{p} \geq j$ then $i+1-\dim R_\mathfrak{p} \geq (n-\dim S_\mathfrak{p})$.  We note $n-\dim S_\mathfrak{p}$ is positive and thus we get $i+1 \geq \dim R_\mathfrak{p}$.  In this case, we note we would merely be examining $\Delta_\mathfrak{p}$.  Thus, we have the Stanley-Reisner ring of $(\Delta ^{(i)})_\mathfrak{p}$ satisfies $(S_\ell^j)$.  Combining these two cases gives the proof of the theorem.

\end{proof}

\section*{Acknowledgment}
I would like to thank my adviser Hailong Dao for insight and mentoring.  I would like to thank Justin Lyle and Mohammad Eghbali for helpful conversations, and I would like to thank my reviewer for detailed, insightful feedback.

\bibliographystyle{amsalpha}
\bibliography{mybib}
\end{document}